\documentclass[]{article}
\usepackage[T1]{fontenc}
\usepackage{amsfonts}
\usepackage{youngtab}
\usepackage{amsmath,amscd,amsthm}
\usepackage{graphicx}
\usepackage{caption}
\usepackage{subcaption}
\newtheorem{assumption}{Assumption}
\usepackage{listings} 
\usepackage{amssymb}
\newcommand{\defeq}{\vcentcolon=}

\usepackage{mathtools}
\usepackage{tikz}
\usetikzlibrary{arrows.meta}
\DeclarePairedDelimiter{\ceil}{\lceil}{\rceil}
\usepackage[linesnumbered,vlined,ruled,commentsnumbered]{algorithm2e}
\usepackage{algorithmic}
\usepackage{graphicx}

\newtheorem{theorem}{Theorem}[section]
\theoremstyle{definition}
\newtheorem{definition}{Definition}[section]
\newtheorem{lemma}{Lemma}[section]

\newtheorem{corollary}{Corollary}[section]
\newtheorem{example}{Example}[section]

\begin{document}
\title{Uncertainty quantification for subgradient descent, with applications to relaxations of discrete problems.}
\author{Conor McMeel\\ Imperial College London
	\and Panos Parpas\\ Imperial College London}
\maketitle              
\begin{abstract}
We consider the problem of minimizing a convex function that depends on an uncertain parameter $\theta$. The uncertainty in the objective function means that the optimum, $x^*(\theta)$, is also a function of $\theta$. We propose an efficient method to compute $x^*(\theta)$ and its statistics. We use a chaos expansion of $x^*(\theta)$ along a truncated basis and study a restarted subgradient method that compute the optimal coefficients.
  We establish the convergence rate of the method as the number of basis functions increases, and hence the dimensionality of the optimization problem is increased.
 We give a non-asymptotic convergence rate for subgradient descent, building on earlier work that looked at gradient and accelerated gradient descent. Additionally, this work explicitly deals with the issue of projections, and suggests a method to deal with non-trivial projections. We show how this algorithm can be used to quantify uncertainty in discrete problems by utilising the (convex) Lovasz Extension for the min s,t-cut graph problem. 

\end{abstract}
\section{Introduction}
	Gradient descent is one of the most fundamental and well-known optimisation methods, and is used throughout industry and other scientific fields in various applications \cite{hochreiter2001learning,zhao2015fast,wang2008method}.
	
	
	In applied contexts, we accept that there will be uncertainty as we use gradient descent \cite{hu2017uncertainty,papi2020uncertainty}. We consider the specific case where our objective function of interest $f(x): \mathbb{R}^q \rightarrow \mathbb{R}$ cannot be evaluated, but we have access to a function $F$ such that:
	\begin{equation}
	    f(x) = \mathbb{E}\left(F(x,v)\right).
	\end{equation}
	We then introduce uncertainty via a parameter $\theta \in \mathbb{R}^d$: we take our normal objective $f(x) : \mathbb{R}^q \rightarrow \mathbb{R}$, and we introduce a parameter of uncertainty $\theta$, itself being $d$-dimensional and arising from a distribution $\pi$. Due to this change, the input $x$ is now a function of $\theta$, which we assume is in $L^2_{\pi}$. This impacts the above Equation like so:
	\begin{equation}
	    f(x(\theta),\theta) = \mathbb{E}\left(F(x(\theta),\theta,v)\right)
	\end{equation}
	There are three sources of uncertainty we are interested in quantifying:
	\begin{enumerate}
	    \item The distribution $\mu$ of the variable $v$ could depend on $\theta$.
	    \item The function $F$ may depend on $\theta$.
	    \item We may not know $\theta$, but assume we know a distribution it comes from due to a priori knowledge of the problem.
	\end{enumerate}
	Its our aim to quantify this uncertainty. We can see this function has a different optimum $x^*(\theta)$ for each value of $\theta$. In our case, we want to learn about $x^*$ as a function of $\theta$, or at least its statistics. Specifically, we want to find a $x^*(\theta)$ such that:
	\begin{equation}
		x^*(\theta) = \arg \min_{x} \pi(f) = \int_{V} F(x(\theta), \theta,v) \mu(\theta, dv).
	\end{equation}
	for $\pi$-almost all $\theta$. Therefore we are looking for a function $x^*(\theta)$ such that for $\pi$-almost all $\theta$ we have that $f(x^*(\theta),\theta)$ is a minimum when we fix $\theta$. One concrete example we will work on is the following:
	\begin{example} \label{exmp::Intro}
	   Let $G = (V,E)$ be an undirected weighted graph, with the parameter $\theta$ representing some uncertainty in the edge weights. Suppose the graph has a source node $s$ and a sink node $t$.
	   
	   For each $\theta$, the min $s-t$ cut problem can be transformed to a convex nonsmooth problem and solved efficiently using the techniques of submodular minimisation \cite{fujishige2005submodular}.
	   
	   We would like to obtain a function $x^*(\theta)$ that can give the min cut for $\pi$ almost all $\theta$, and to obtain statistics on $x^*(\theta)$.
	\end{example}
	For other examples of this problem, we direct the reader to \cite{sullivan2015introduction}.
	
    Similarly to previous methods \cite{mcmeel2021uncertainty,crepey2020uncertainty}, we will assume $x^*(\theta)$ is continuous and decompose it over a basis of $L^2_\pi$ and then find the basis coefficients. This framework will be used for finding our general results. In a subsequent section where we study the generalised version of the problem posed in Example \ref{exmp::Intro}, we will see that this assumption on $x^*(\theta)$ fails in general, and discuss numerous methods for solving the problem in this case.
    
    A naive method for solving this problem may be the following: repeatedly sample $\theta \sim \pi$, and for each sample $\theta_i$ consider the induced function $f_{\theta_i}: \mathbb{R}^d \rightarrow \mathbb{R}$ we get by fixing $\theta = \theta_i$. We can then find the optimum from gradient descent, and repeat this to obtain desired statistics of $x^*(\theta)$ from a Monte Carlo method. 
	
	To see the effects of this method in terms of complexity, we use the justification of \cite{mcmeel2021uncertainty}, where the authors consider an example where they seek to compute the mean of $x^*(\theta)$ with respect to $\theta$. Namely, to compute the integral:
	\begin{equation}
	    \int_{\Theta} x^*(\theta) \pi(d\theta)
	\end{equation}
	
	Standard Monte Carlo results show that the estimation of the error decreases with rate $\sqrt{N}$, where $N$ is the number of times subgradient descent is ran. This means to reduce our error by a factor of $1/2$, we will need to run gradient descent $4$ times as much, which quickly becomes infeasible, and motivates us to consider our method above, which will solve the problem for all $\theta$ with only one run.
	
	\subsection{Our Contributions}
	We propose and analyse an uncertainty quantification algorithm for the subgradient method in \cite{yang2018rsg}, where $f$ is convex but not strongly convex and non-smooth. Such uncertainty quantification algorithms have been considered previously in a Stochastic Approximation setting \cite{crepey2020uncertainty}, where converge is shown asymptotically, and for gradient descent in the strongly convex case \cite{mcmeel2021uncertainty}, where an optimal rate was shown. Our methods extend the ideas of \cite{mcmeel2021uncertainty} to the case where $f$ is non strongly convex and nonsmooth, and shows the following:
	\begin{itemize}
	    \item Non-asymptotic convergence rates once the algorithm runs for ``long enough'' that are the same as subgradient descent (Theorem \ref{theorem::mainresult}), as studied in \cite{yang2018rsg}.
	    \item We also give a convergence rate to the point where we have run the algorithm ``long enough'' (Lemma \ref{lemma::LinConvRemainder}).
	    \item We study the case where a particular local error bound \cite{zhou2017unified} holds, and show that for problems of this type, that a linear rate holds, similarly to \cite{yang2018rsg}. This problem class includes submodular optimisation.
	    \item Previous methods have found it difficult to implement projections. Our method, for example, can easily implement a box constraint on the function value $f(x(\theta),\theta)$.
	\end{itemize}
	We show that for our method, the rate of convergence to the function value to an $\varepsilon$-neighbourhood of the true solution achieves the usual rate after sufficiently many iterations, and we discuss how many iterations are needed for this to happen. We conclude by performing some experiments that demonstrate our method, and then with a brief discussion.
	
\subsection{Previous Work}
	Much of the previous work surrounding chaos expansion for Uncertainty Quantification with gradient methods have focused on Robbins-Monro \cite{robbins1951stochastic} stochastic approximation algorithms. 
	
	Turning our attention to the previous body of uncertainty quantification for SA algorithms, in \cite{kulkarni2009finite} a similar spectral approach is taken, but truncated to a finite dimension at all iterations. They truncate by letting $x(\theta) = u_iB_i$, for some finite family of functions $B_i$, and then perform a standard Stochastic Approximation procedure to calculate the coefficients $u_i$. After this, error analysis from the finite-dimensional approximation is performed.
	
	However, we are interested in convergence, and thus will need to let the number of basis functions tend to infinity. In terms of infinite-dimensional methods, \cite{yin1990h} give an SA algorithm in a Hilbert space. However, the algorithm is defined in infinite-dimensional space, so is not practically implementable.
	
	In our work we will increase the number of basis functions over time, known as a sieve method. In \cite{crepey2020uncertainty}, they study a SA sieve algorithm. They show asymptotic convergence assuming only standard SA local separation and a local strong convex type assumption. However, there are also conditions on the step size which relate to the number of basis functions and the decay of the truncation error, and its our hope that in our setting we will not need conditions like this.
	
	In \cite{mcmeel2021uncertainty}, this sieve method was used for the first time in a non-stochastic approximation setting. A similar method to \cite{crepey2020uncertainty} was introduced for gradient descent for strongly convex functions, which allowed the use of standard gradient descent step sizes. It was shown that after sufficiently many step sizes were taken, there was convergence to the optimal point at an optimal accelerated rate.
	
\section{Preliminaries} \label{section::AlgoIntro}
	In this section, we introduce our core algorithm for uncertainty quantification for subgradient descent, and some useful mathematical preliminaries. We now give a formal statement of our problem:
		\begin{definition}[UQ Problem]
		Let $f$ be a function defined on $L^2_\pi \times \mathbb{R}^d \rightarrow \mathbb{R}$, that takes in a $d$-dimensional real parameter $\theta$, and function $x \in L^2_{\pi}$ evaluated at $\theta$. Suppose we cannot access the function $f$, but instead a function $F$ such that $f(x(\theta),\theta) = \mathbb{E}_v\left(F(x(\theta),\theta,v)\right)$. Assume the function is Lipschitz continuous, but not necessarily smooth. Let the variable $\theta$ come from some distribution $\pi$. 
		
		Then the Uncertainty Quantification problem asks us to do the following two things:
		\begin{enumerate}
		    \item Find a function $x^*(\theta)$ such that $\int_{V} F(x(\theta), \theta,v) \mu(\theta, dv)$ is minimised for $\pi$-almost all $\theta$.
		    \item To learn the distribution of $x^*(\theta)$ when $\theta \sim \pi$, or at least compute its statistics.
		\end{enumerate}
	\end{definition}
	As in \cite{mcmeel2021uncertainty}, we suppose that $x^*(\theta) \in L^2_\pi$, and express it as an infinite series in a chosen basis of that space:
	\begin{equation*}
	    x^*(\theta) = \sum_{i} u_i B_i(\theta),
	\end{equation*}
    and we see that our problem is reduced to finding the optimal $u_i$. We will find these coefficients using subgradient descent. To make our computations tractable, we will at each iteration $k$ fix a level $m_k$ above which we set all $u_i$ to zero, and only update nonzero $u_i$. By letting $m_k \rightarrow \infty$, we will show convergence. We note that to avoid ambiguity, we insist that $m_k$ is monotonic increasing in $k$.
    
    As previously stated we will be solving for the vector of coefficients $u$ as a proxy for finding $x^*$ itself. Therefore, we must define a new operator that will be used as the subgradient. We first recall a subgradient, for a function $f: \mathbb{R}^n \rightarrow \mathbb{R}$ at $x$ is defined as satisfying the following for all $y$:
    \begin{equation}
        f(y) \geq f(x) + \langle g_x, y - x \rangle_2.
    \end{equation}
    Note that by rearranging this Equation, we find:
    \begin{equation}
        f(y) - f(x) \geq \langle g_x, y-x \rangle_2.
    \end{equation}
    We want to adapt this definition to our function type in the manner of \cite{mcmeel2021uncertainty}:
    \begin{definition} \label{defn::subgradient}
        We say that a subgradient for $f: L^2_\pi \rightarrow \mathbb{R}$, $g(\theta)$ is also an element of $L_\pi^2$, and that it is defined by satisfying the following for $x(\theta)$, for all $y(\theta)$ and all $\theta$:
        \begin{equation}
            \left(f(y(\theta),\theta) - f(x(\theta),\theta)\right) \geq \langle g_x(\theta), y(\theta) - x(\theta) \rangle_2, \label{eqn::SubgradDef}
        \end{equation}
        where the two norm indicates we are considering the induced real numbers $g_x(\theta)$, $x(\theta)$, $y(\theta)$ for a fixed $\theta$ rather than elements of $L^2_\pi$. The subdifferential at $x(\theta)$ is the set of subgradients at $x(\theta)$.
    \end{definition}
    For us this subgradient will not be immediately useful in its current form, as we only want something that operates on the first $m$ basis functions of $x(\theta)$ in its $L^2$ basis. To this end, we define a \textit{truncated subgradient} as follows:
    \begin{definition}
        Let $g(\theta)$ be a subgradient as in Definition \ref{defn::subgradient}. Write its $L^2_\pi$ basis representation as so:
        \begin{equation*}
            g(\theta) = \sum_i u_i B_i(\theta)
        \end{equation*}
        then the level $m$ truncated subgradient is defined as:
        \begin{equation*}
            g_m(\theta) = \sum_{i=1}^m u_i B_i(\theta)
        \end{equation*}
        and the level $m$ subdifferential at a point, $\partial_m f(x(\theta),\theta)$, consists of all of the level $m$ truncated subgradients at $f(x(\theta),\theta)$.
    \end{definition}
    We note that the integral $\int_{\Theta} g(\theta)B_i(\theta) \pi(\theta, dv)$ gives the coefficients $u_i$. We also note that a level $m$ truncated subgradient satisfies the subgradient definition for all $x(\theta), y(\theta) \in L^2_{m,\pi}$.
    
    When the algorithm is introduced in the next section, we will discuss in more detail the computation of the subgradient. For now, we note that we will require it to be an unbiased estimate along with a standard bound on the variance.

	In this work, we are considering convex functions. As our function has the $L^2_\pi$ space as input, we give a specific definition of what convexity means in this context:
	\begin{definition}
		Consider a function $f: L^2_\pi \times \mathbb{R}^d \rightarrow \mathbb{R}$. We say that $f$ is convex if for all subgradients $g_{x_1}(\theta), g_{x_2}(\theta)$ of $x_1, x_2$, we have:
		\begin{equation*}
		    \langle \nabla g(x_1(\theta),\theta)- \nabla g(x_2(\theta),\theta), x_1(\theta) - x_2(\theta) \rangle_\pi \geq 0
		\end{equation*}
		for all $x_1(\cdot), x_2(\cdot) \in L^2_\pi$.
	\end{definition}
    Additionally, while our objective function will not be smooth, we will assume that it is Lipschitz continuous, but here we want to specify that we are using the $\pi$ norm also, and only considering the first argument. Formally what we mean is: 
	\begin{definition}
		Let $f: L^2_\pi \times \mathbb{R}^d \rightarrow \mathbb{R}^q$. We say that $f$ is Lipschitz Continuous with parameter $L$ if:
		\begin{equation*}
			||f(x_1(\theta), \theta) - f(x_2(\theta), \theta)||_\pi \leq L||x_1(\theta) - x_2(\theta)||_\pi
		\end{equation*}
		for all $x_1(\cdot), x_2(\cdot) \in L^2_\pi$.
	\end{definition}
	
    \subsection{Submodular Minimisation}
    In this subsection, we give an introduction to submodular minimisation. We are interested in applying this algorithm to graph cut problems, which we will see is an instance of submodular minimisation.
    
    We begin by giving the definition of a submodular function:
    \begin{definition}
        Let $V$ be a set, and $f:V \rightarrow \mathbb{R}$ be a set function. We call $f$ submodular if the following holds for all $X, Y \in V$:
        \begin{equation*}
            f(X) + f(Y) \geq f(X \cup Y) + f(X \cap Y)
        \end{equation*}
    \end{definition}
    It can be shown that the min $s,t$ cut is an example of a submodular function \cite{fujishige2005submodular}. Submodular functions can be minimised in polynomial time, and a number of combinatorial algorithms exist for this purpose \cite{iwata2009simple,iwata2001combinatorial}.
    
    Minimising submodular functions can also be done by using a convex relaxation of $f$ called the \textit{Lovasz Extension}, which we define now:
    \begin{definition}
        Let $f$ be a submodular function. Then we define the Lovasz Extension $f^L: [0,1]^{|V|} \mathbb{R}$ as follows:
        \begin{equation*}
            f^L(x) = \sum_{i=0}^{|V|} \lambda_i f(S_i)
        \end{equation*}
        where $\emptyset = S_0 \subset S_1 \subset \ldots \subset S_n = V$ is a chain such that $\lambda_i \geq 0$, $\sum_i \lambda_i = 1$, and $\sum_{i=0}^{|V|} \lambda_i \textbf{1}_{S_i} = x$.
    \end{definition}
    It can be shown that if $f$ is submodular, then the Lovasz Extension is convex and piecewise linear. Furthermore, if $x^*$ is a minimum point of the Lovasz Extension, it can be converted to a minimum of the original function via a $\phi$-rounding procedure:
    \begin{lemma}[$\phi$-rounding procedure]
        Suppose that $x^*$ is a $\varepsilon$-optimum of the Lovasz Extension. Then define the following family of subsets of $V$:
        \begin{equation}
            X_\phi = \{\cup e_i \quad | \quad x^*_i \geq \phi\}.
        \end{equation}
        By taking the maximum of $f(X_\phi)$ over $\phi \in [0,1]$, it can be shown we will find a subset $X_{\phi^*}$ that is also an $\varepsilon$-optimum of the original function $f$ \cite{fujishige2005submodular}.
    \end{lemma}
    Additionally, we note that the minimiser of $f$ is also a minimiser of $f^L$ \cite{fujishige2005submodular}.
    
	\section{Algorithm Formulation}
	Now that we have covered the mathematical preliminaries, we detail our algorithm for subgradient descent. Before we begin, we would like to briefly motivate our UQ methods by comparing the complexity in terms of number of Monte Carlo samples. We will consider the motivating example of using Jacobi polynomials in the case where the dimension of $x$ is $1$ and of $\theta$ as $d$ as in \cite{crepey2020uncertainty}.
	
	To begin, lets consider a more naive approach where we truncate to a fixed level $m$, and then perform gradient descent a large number of times in order to construct a Monte Carlo approximation to the vector $u$. As outlined in \cite{crepey2020uncertainty}, in order to balance the error of the truncation and the actual approximation, we must have the number of Monte Carlo samples to approximate the $u^*_i$ increase as $\varepsilon^{-1+d/(2(\eta-1))}$, where $\eta$ is the order of differentiability of $x^*(\theta)$. Note that for us, each Monte Carlo sample must be obtained by running gradient descent, so this is the amount of times gradient descent has to be run to convergence, and we suffer from the curse of dimensionality for large $d$.
	
	We note that in practice, we will not be able to compute the subgradient $g_x(\theta)$ exactly, and so we must estimate it by $g_x'(\theta)$. While we will suggest a procedure later, we emphasise that our convergence analysis is independent of the proceudre we use, so long as it satisfies some basic bias and variance properties.
	
	In our case, we will show that we converge linearly to the correct vector $u^*$, and there are no nested computations that require many runs of gradient descent like the method above.
	
	We now move on to discussing our algorithm, which is presented in Algorithm \ref{alg::algo_full}. As of yet, we have not discussed how to estimate the vector $g_x'(\theta)$. The most natural is to perform a Monte Carlo procedure to approximate the integrals $\int_{\theta} g_x'(\theta) B_i(\theta) \pi(d\theta)$. This is the procedure we use in our experiments, though we emphasise again that other procedures could be used. We note that an error analysis of this procedure can be found in \cite{mcmeel2021uncertainty}, where the proof follows by instead of having gradients, choosing one representative subgradient at each point $g_x$.
	
	We then form a subgradient to update the coefficients $u_i$. We will later show that the distance between the optimal function value we update for the previous representation and the true $x^*(\theta)$ converges to zero. After every loop, we will change $m_k$ in such a manner that $m_k \rightarrow \infty$ in a monotonically increasing manner. As approximating $x(\theta)$ is equivalent to approximating the vector $u$, we will from here on phrase our algorithm as attempting to find the optimal vector $u^*$.
	
	This procedure can be thought of as performing iterations on a family of optimisation problems that converge to the ``correct'' problem, in an effort to save on complexity while still eventually converging to the right point. In later sections, we will see an experimental verification of this, but additionally we will see that this procedure also functions as variance reduction on the output.

\begin{algorithm} \label{alg::subroutine}
    \caption{Subgradient subroutine}
    \label{algo_sub}
        \begin{algorithmic}
        \STATE{\textbf{Input}: Function $f(x, \theta)$, a number of iterations per stage $T$, a step size $\eta$, a number of basis functions to use $m_k$, an initial point $u^0$.}
        \STATE{\textbf{Output}: A function $x^T(\theta) \in L^2_\pi$.}
        \STATE{Define $x^{0} = \mathcal{I}(u^{0})$}
        \FOR{$t = 1, \ldots, T$}
        \STATE{Query subgradient oracle to obtain approximate subgradient $g'_{m_k}(x^{t-1}(\theta))$.}
        \STATE{Update $x^t(\theta) = \Pi\left(x^{t-1}(\theta) - \eta g'_{m_k}(x^{t-1}(\theta))\right)$}
        \ENDFOR
        \STATE{Let $\tilde{x}^T(\theta) = \sum_{i=1}^T \frac{x^i(\theta)}{T}$, and $u_T = \mathcal{I}(\tilde{x}^T(\theta)$.}
        \RETURN $u_T$
        \end{algorithmic}
\end{algorithm}

\begin{algorithm} 
    \caption{RSG subroutine}
    \label{alg::subroutineRSG}
        \begin{algorithmic}
        \STATE{\textbf{Input}: Function $f(x, \theta)$, a number of stages in this loop $K$, the number of iterations per stage in this loop $t$, $\alpha > 1$, a sequence of numbers of basis functions to use $m_k$, an initial point $u^{I}$ with only the first $m_0$ entries possibly nonzero.}
        \STATE{\textbf{Output}: A function $x^{F}(\theta) \in L^2_\pi$.}
        \STATE{Let $u_0 = u^I$}
        \STATE{Set $\eta_1 = \frac{\varepsilon_0}{\alpha \left(G^2+V^2\right)}$.}
        \FOR{$k = 1, \ldots, K$}
        \STATE{Call subroutine SG to obtain $u_k = SG(u_{k-1}, \eta_k, t)$.}
        \STATE{Set $\eta_{k+1} = \eta_{k}/\alpha$.}
        \ENDFOR
        \RETURN $x^{F}(\theta) = \mathcal{I}(u_K)$
        \end{algorithmic}
\end{algorithm}

\begin{algorithm} 
    \caption{Restarted Subgradient method}
    \label{alg::algo_full}
        \begin{algorithmic}
        \STATE{\textbf{Input}: Function $f(x, \theta)$, a number of stages per loop $K$, the number of iterations per stage in loop $j$ $t_j$, $\alpha > 1$, a sequence of numbers of basis functions to use $m_k$, an initial point $u^{I}$ with only the first $m_0$ entries possibly nonzero.}
        \STATE{\textbf{Output}: A function $x^{F}(\theta) \in L^2_\pi$.}
        \STATE{Let $u_0 = u^I$}
        \STATE{Let $\varepsilon_0$ be a valid lower bound on optimality for $||f(x^0(\theta),\theta)||_\pi$}
        \FOR{$i = 1,2,\ldots$}
        \STATE{Let $u^i = RSG(u^{i-1}, f, K, t_i, \alpha, m, \varepsilon_0)$}
        \STATE{If termination condition reached, break}
        \ENDFOR
        \RETURN $x^{F}(\theta) = \mathcal{I}(u_K)$
        \end{algorithmic}
\end{algorithm}
The algorithm is presented in Algorithm \ref{alg::algo_full}. We note that there is one more round of resetting than in the presentation in \cite{yang2018rsg}. We will later see that this extra layer of resetting is needed to ensure convergence as we vary the number of basis functions. In the analysis that follows, we refer to the loop indexed by $i$ as the \textit{outer loop}, and the loop indexed by $K$ as the \textit{inner loop}.

\section{Generic Convergence for Subgradient UQ} \label{sec::GenSubUQ}
To show convergence for our subgradient UQ (given in Algorithm \ref{alg::algo_full}), we follow the approach in \cite{mcmeel2021uncertainty}: we first show that if we remain at a fixed level of basis functions $m_k$, that the method converges to some function value. We will refer to this value as a \textit{level $m_k$ optimum}.

We then show that as $m_k \rightarrow \infty$, the level $m_k$ optimal value converges to the true optimal value. Finally, we show that when the number of basis functions is high enough, we converge to the true optimum with the same rate as in the fixed level case.
\subsection{Fixed Level Convergence}
We follow the proof strategy in \cite{yang2018rsg} here. They operate on some core assumptions for the full restarted subgradient (RSG) algorithm, which we adapt to our UQ setting. We also add some assumptions about the estimation of the subgradient:
\begin{assumption}
    For the problem of minimising convex $f$, we assume:
    \begin{itemize}
        \item For any $x(\theta) \in \Omega$, we know a constant $\varepsilon_0 \geq 0$ such that $||f(x(\theta),\theta) - f(x^*(\theta),\theta)||_2 \leq \varepsilon_0$.
    \end{itemize}
    Additionally, let $g'_x(\theta)$ be the estimation of a subgradient $g_x(\theta)$ at $x(\theta)$. Then we assume the following:
    \begin{itemize}
        \item $\mathbb{E}\left(g'_x(\theta)\right)$ is a subgradient at $x(\theta)$.
        \item $\mathbb{E}\left(g'_x(\theta)^2\right) \leq \mathbb{E}\left(g_x(\theta)^2\right) + V^2 \leq G^2 + V^2$.
    \end{itemize}
\end{assumption}
The first assumption is a natural extension from the usual subgradient method case, to the case where the input arguments are $L^2_\pi$ functions. The second set are similar to those used in \cite{mcmeel2021uncertainty} for the bias and variance properties of the estimation of the gradient. For the next assumption, we must define some technical machinery first.

In all that follows, fix a number of basis functions $m$. Let the set of level $m$ optimal elements $x^*_m(\theta) \in L^2_{m,\pi}$ be denoted $\Omega^*_m$. Next, let the element of $\Omega^*_m$ that is closest to some arbitrary $x(\theta) \in L^2_\pi$ (measured in the $\pi$-norm) be denoted $y^*_m(x(\theta))$. Note that it is uniquely defined as the set $\Omega^*_m$ is convex, as is the $\pi$ norm.

Now we define $L_{\varepsilon,m}$ and $S_{\varepsilon,m}$ as the $\varepsilon$-level set and sublevel sets at $m$ basis functions respectively, where distance is again measured in the $\pi$ norm. Namely, if the optimal function value at $m$ basis functions is $f^*_m(\theta)$, then $S_{\varepsilon,m}$ consists of those elements $x(\theta) \in L^2_{m,\pi}$ such that $||f(x(\theta), \theta) - f^*_m(\theta)||_\pi \leq \varepsilon$. From here, we can define $B_{\varepsilon,m}$ as the maximum distance between $L_{\varepsilon,m}$ and $\Omega^*_m$, that is:
\begin{equation}
    B_{\varepsilon,m} = \max_{x(\theta) in L_{\varepsilon,m}} \min_{y_m(x(\theta)) \in \Omega^*_m} ||x(\theta) - y(\theta)||_\pi = \max_{x \in L_{\varepsilon,m}} ||x(\theta) - y^*_m(x(\theta))||_\pi.
\end{equation}
This leads us to our second assumption:
\begin{assumption}
    $B_{\varepsilon,m}$ is finite $\forall$ $k$.
\end{assumption}
Introducing the quantity $B_{\varepsilon,m}$ is key to the specialised analysis of polyhedral convex optimisation (and hence, submodular minimisation) we will do later. Additionally, we extend the definitions to infinite $m$ by suppressing the $m$ subscript. We similarly assume that $B_{\varepsilon}$ is finite.

With this, we next define the following function:
\begin{equation*}
    w^{\dagger}_{\varepsilon,m}(\theta) \defeq \arg \min_{u \in \mathcal{S}_{\varepsilon,m}} ||u(\theta) - w(\theta)||_\pi^2
\end{equation*}
which is the closest function to $w$ in the $\varepsilon$-sublevel set for $m$ basis functions.

Note firstly that the sublevel set is convex, and $w(\theta)$ is a continuous function of $\theta$. Because of these two properties, we have that $w^\dagger_{\varepsilon,m}$ is a continuous function of $\theta$. We next adapt the definition of the normal cone and the first order residual from \cite{yang2018rsg} to our setting as follows:
\begin{definition}
    Let $x(\theta) \in L^2_{m,\pi}$. We define the level $m$ normal cone of $L^2_\pi$ at $x(\theta)$ to be:
    \begin{equation*}
        \mathcal{N}_m(x(\theta)) \defeq \{ y(\theta) \in L^2_{m,\pi} : \langle y(\theta), z(\theta) \rangle_\pi \leq \langle y(\theta), x(\theta) \rangle_\pi, \forall z(\theta) \in L^2_{m,\pi} \}
    \end{equation*}
    With this definition, we define the level $m$ first order residual to be:
    \begin{equation*}
        \text{dist}\left(0, f(x(\theta),\theta) + \mathcal{N}_m(x(\theta))\right) \defeq \min_{g \in \partial_m f(x(\theta),\theta), v \in \mathcal{N}_m(x(\theta))} ||g(\theta) + v(\theta)||_\pi
    \end{equation*}
    As before, we can extend these definitions to the infinite-dimensional case by suppressing the subscripts $m$.
\end{definition}
As a motivation for the term \textit{first order residual}, we note that in the infinite-dimensional case, an element of $L^2_\pi$ is optimal if and only if the first order residual at that point is $0$. The final constant we need to define here is the following:
\begin{definition}
    For any $\varepsilon > 0$ such that $\mathcal{L}_\varepsilon \neq \emptyset$, we have:
    \begin{equation*}
        \rho_{\varepsilon,m} = \min_{x(\theta) \in \mathcal{L}_{\varepsilon,m}} \text{dist}\left(0, f(x(\theta),\theta), \mathcal{N}_m(x(\theta))\right)
    \end{equation*}
    where again as before, we allow the infinite-dimensional case by suppressing the subscript $m$.
\end{definition}
We can now give a result on the performance of the subroutine, Algorithm \ref{algo_sub}. A proof can be found in the Appendix:
\begin{lemma} \label{lem::constStepSG}
    Suppose we run the subroutine Algorithm \ref{algo_sub}, but we only have noisy subgradients as defined above. Using a constant step size $\eta$ for $T$ iterations, we have:
    \begin{equation*}
        \mathbb{E}\left(\mathbb{E}_\pi\left(f(\tilde{x}_T(\theta),\theta) - f(x^*(\theta),\theta) | \mathcal{F}_k\right)\right) \leq \frac{\left(G^2 + V^2\right) \eta}{2} + \frac{||x_1(\theta) - x(\theta)||_\pi^2}{2\eta T}
    \end{equation*}
    where $\tilde{x}_T = \frac{1}{T}\sum_{i=1}^T x_i$, the outer expectation is with respect to the filtration $\mathcal{F}_k$, and $x^*(\theta)$ is some truncated optimum.
\end{lemma}

From here, we can give a proof for the restarted subgradient method. We note that in this scenario, we only need to use one outer loop.
\begin{theorem} \label{theorem::subgradientfixed}
    Suppose the previous assumptions all hold, and we fix $m$ basis functions. Let $t \geq \frac{\alpha^2 \left(G^2 + V^2\right)}{\rho_{\varepsilon,m}^2}$ and $K = \lceil \log_{\alpha} \left( \frac{\varepsilon_0}{\varepsilon}\right) \rceil$ in our algorithm, with at most $K$ stages and only one outer loop. Then we return a $2\varepsilon$ optimal solution in expectation in at most $tK$ iterations.
\end{theorem}
In the above Theorem, we have used the quantity $\rho_{\varepsilon,m}$, and we now give a Lemma that will relate this quantity to $B_{\varepsilon,m}$. This relation will be key to the analysis on submodular minimisation problems. A proof is found in the appendix.
\begin{lemma} \label{lem::Bep}
    For any $\varepsilon > 0$ such that $L_{\varepsilon,m} \neq 0$, we have $\rho_{\varepsilon,m} \geq \frac{\varepsilon}{B_{\varepsilon,m}}$, and for any $w \in L^2_{\pi,m}$:
    \begin{align*}
        &||w(\theta) - w_{\varepsilon,m}^{\dagger}(\theta)||_\pi \\
        &\leq \frac{||w_{\varepsilon,m}^{\dagger}(\theta) - w_{\varepsilon,m}(\theta)^*||_\pi}{\varepsilon}\mathbb{E}_\pi\left(f(w(\theta),\theta) - f(w_{\varepsilon,m}^\dagger(\theta),\theta)\right)\\
        &\leq \frac{B_{\varepsilon,m}}{\varepsilon}\mathbb{E}_\pi
        \left(f(w(\theta),\theta) - f(w_{\varepsilon,m}^\dagger(\theta),\theta)\right)
    \end{align*}
    This result also holds in the infinite-dimensional case by removing all $m$ subscripts.
\end{lemma}
From here, we can state the following corollary:
\begin{corollary} \label{cor::FixedLevelGen}
    Suppose our first two assumptions hold. Then the iteration complexity for obtaining a $2\varepsilon$ level $m$ optimal solution is $\mathcal{O}\left(\frac{\alpha^2 \left(G^2 + V^2\right) B_{\varepsilon,m}^2}{\varepsilon^2} \ceil{\log_{\alpha} \left( \frac{\varepsilon_0}{\varepsilon}\right)} \right)$ provided $t, K$ are as before.
\end{corollary}

\subsection{Convergence of the Optimum}
In the previous subsection, we are considering a series of optimisation problems where the subgradient $g'_m$ converges to the correct subgradient $g$ as $m \rightarrow \infty$. We would also like to confirm if a similar property holds for $f(x^*_m(\theta),\theta)$, or on the set of its subgradients.

In this subsection, we will show both that our convexity is sufficient for this convergence, and also give an explicit convergence rate in terms of the level $m$ remainder $R_m$.
	
We mentioned previously that in general, the truncated optimum is not equal to the projection of the true optimum to the same amount of basis functions. We now show that the distance between the two converges to zero and interestingly, we see its convergence properties are related to the Lipschitz constant of the function:
	\begin{lemma} \label{lemma::LinConvRemainder}
		Let $u^*_m$ be level $m$ optimum , and let $x^*_m(\theta) = \mathcal{I}(u^*_m)$. Then the quantity $||P_{m}(u^*) - u^*_m||$ converges at a rate no slower than that of $R_{m}$. In particular, we have:
		\begin{equation*}
		||f(x^*_m(\theta)) - f(x^*(\theta))||_\pi \leq L ||R_{m}||_2^2,
		\end{equation*}
		where $\kappa$ is the condition number of $f$. 
	\end{lemma}
	Note that $R_{m}$ will converge to zero with some basis dependent rate, and so we have that same rate of convergence for $||f(x^*_m(\theta)) - f(x^*(\theta))||_\pi$.
	\begin{proof}
		Recall that $f$ is convex. We will use $g_m$ to refer to a level $m$ subgradient, and $g$ to refer to a full subgradient. We use the first order definition for convexity:
		\begin{equation*}
		    f(x) \geq f(y) + \langle \nabla f(y), x-y \rangle
		\end{equation*}
		with $x = x^*(\theta), y = x^*_m(\theta)$, to find upon rearranging, noting $f(x) - f(y) \leq 0$ and taking $\pi$ norms:
		\begin{equation*}
		    ||f(x^*_m(\theta)) - f(x^*(\theta))||_\pi \leq \langle g(x^*_m(\theta)), x^*_m(\theta) - x^*(\theta) \rangle \leq L ||R_m||
		\end{equation*}
		where in the second inequality, we have used Cauchy-Schwarz, and the fact that as $f$ is Lipschitz, its subgradients are bounded in norm.
	\end{proof}
\subsection{Overall Convergence Rate}
We are now in a position to demonstrate overall convergence of our method. Our high-level strategy will be as follows:
\begin{enumerate}
    \item Once we have reached a sufficiently high $m_k = M$, then an inner loop taken with more than $M$ basis functions will look very similar to those taken with more than $M$ basis functions.
    \item Furthermore, the level $M$ optimum is a small distance away from the true optimum.
    \item The first outer loop taken where we always have $m_k \geq M$ will exhibit the required convergence.
\end{enumerate}
The result is stated formally as follows:
\begin{theorem} \label{theorem::mainresult}
    Let $||f(x^*_M(\theta),\theta) - f^*||_\pi \leq \varepsilon_1$ for a desired $\varepsilon_1$ and some $M$, and consider Algorithm \ref{alg::algo_full}. For each inner loop where we use $m_k$ iterations, let $\rho = \min_{m_1,\ldots,m_k} \left(\rho_{\varepsilon,m_k}\right).$ Let $t_i \geq \frac{\alpha^2 \left(G^2 + V^2\right)}{\rho^2}, K_i = \lceil \log_\alpha \left( \frac{\varepsilon_0}{\varepsilon} \right) \rceil$.
    
    Let $i^*$ be the first outer loop which has $m_k \geq M$ $\forall k$. Then we obtain a $3\varepsilon$-accurate solution at the conclusion of outer loop $i^*$.
\end{theorem}

\section{A special class - submodular minimisation}
In this section we use local error bounds to give specialised results on the quantity $B_{\varepsilon,m}$, as we do not know it in general. We will see that this gives us a linear convergence rate for a class of problems that includes submodular minimisation. We edit the definition of the local error bound slightly to fit our UQ setting:
\begin{definition} \label{defn::localerrorbound}
    Say that $f$ admits a local error bound on the level $m$ $\varepsilon$-sublevel set $S_{\varepsilon,m}$ if:
    \begin{equation}
        ||w(\theta) - w^*(\theta)||_\pi \leq c_m||f(w(\theta),\theta) - f^*_m(\theta)||_\pi^{\phi_m}, \forall w \in S_{\varepsilon,m}
    \end{equation}
    where $w^*$ is the closest point in $\Omega^*_m$ to $w$, $0 < \phi_m \leq 1$, and $0 < c_m < \infty$ all constants. We as usual deal with the infinite-dimensional case by removing subscripts $m$.
\end{definition}
Note that this definition implies that $B_{\varepsilon,m} \leq c_m e^{\phi_m}$. We now give a polyhedral error bound condition, and claim that submodular minimisation satisfies this:
\begin{lemma}
    Suppose $\Omega$ is a polyhedron and the epigraph of $f$ is also a polyhedron. There exists a constant $\kappa > 0$ such that:
    \begin{equation}
        ||w - w^*||_\pi \leq \frac{||f(w) - f_*||_\pi}{\kappa}, \forall w \in \Omega
    \end{equation}
    That is, $f$ as a local error bound on $S_\varepsilon$ with $\theta = 1$ and $c = 1/\kappa$.
\end{lemma}
\begin{lemma}
The Lovasz Extension satisfies the hypotheses of the previous Lemma. The Lovasz Extension truncated to a finite level $m$ also satisfies this Lemma.
\end{lemma}
\begin{proof}
    In this case, we have that the function $f$ is also convex for each value of $\theta$, when we fix $\theta$ and treat the function as a real-valued function.
    
    Because of this, we can simply take the proofs of the corresponding Lemmas in \cite{yang2018rsg} and take the $\pi$ norms of the results to reach our conclusion.
\end{proof}
We can now substitute in the constants we've derived to Corollary \ref{cor::FixedLevelGen} to get the final result on the fixed level convergence:
\begin{corollary} \label{cor::NewLinear}
    Suppose the first two Assumptions hold. For minimising the Lovasz Extension, the iteration complexity for a $\varepsilon$ optimal solution (in terms of function value) is $\mathcal{O}\left(\frac{\alpha^2 \left(G^2+V^2\right) }{\kappa^2} \lceil \log_{\alpha} \left( \frac{\varepsilon_0}{\varepsilon}\right) \rceil \right)$ provided $t = \frac{\alpha^2 \left(G^2+V^2\right)}{\kappa^2}$, $K = \lceil \log_{\alpha} \left( \frac{\varepsilon_0}{\varepsilon}\right) \rceil$. By setting $\varepsilon = \varepsilon\kappa$, we can also find an $\varepsilon$-optimal solution in terms of the iterate.
\end{corollary}
where the corresponding asymptotic convergence, and eventual convergence with the above rate follows from the generic analysis.

\subsection{Rounding to a discrete set}




In general we note that as we will not have a unique minimum, and in fact this procedure will not even output the minimum with the least or most elements for a fixed $\theta$, it is a little ambiguous what information we want to gain from the output of our algorithm $x^*(\theta)$. We propose a rounding procedure that will convert $x^*(\theta)$ into a discrete set $S(\theta)$ with the following properties: so long as our Lovasz Extension has converged within $\varepsilon$ of the true optimum, then for each fixed $\theta$, every element of the set $S(\theta)$ belongs to a minimiser of the submodular function $f$.

To do this, suppose we run RSGUQ until $\varepsilon$ convergence. Then taking a threshold of $1-\varepsilon$, we add the element $e$ to $S(\theta)$ when its corresponding continuous variable is above $(1-\varepsilon)$. 

Note that if we use a constant threshold in this manner, we will not achieve an actual minimiser as a set in general, even if we are arbitrarily close to the minimum, as two elements $x_i(\theta), x_j(\theta)$ that are equal at the optimum may be extremely close together at our approximate optimum, but on either side of the threshold.

\subsection{Challenges and a new basis set}
All of this has been done to gain information about $x^*(\theta)$, where we are solving a submodular minimisation problem. In this case, the projection we must solve is that $x(\theta) \in [0,1]$ for all $\theta$. This projection is non-trivial in general, its hard even to find the maximum of a Fourier Series.


As the projection in this case is difficult and the main obstacle to the methods speed, we offer an alternative to computing $x^*(\theta)$. We will refer to this as the piecewise constant method: in the case that $\theta$ is one-dimensional, we can decompose the domain of $\theta$ into progressively more intervals, and write $x^*(\theta)$ as a piecewise constant function. This has the advantage of making the required projection trivial, but the disadvantage of being difficult to scale up in dimensions of $\theta$ in a way that we can still control the number of basis functions effectively.


In our experiments, we will use the piecewise constant method. First however, we give some convergence results for it that are analogous to Section \ref{sec::GenSubUQ}.

\subsection{Piecewise Constant Method} \label{miniSec::PiecewiseConst}
Until now, the way we deal with increasing $m_k$ in the RSGUQ method has been trivial - we simply add the next basis function to our representation for $x$, and keep all basis functions we have. However, in the piecewise constant, we are removing one basis function and replacing it with two new ones at each step. This allows us to retain orthogonality and control the speed of increasing $m_k$ easily.

\begin{algorithm} 
    \caption{Piecewise constant method - increasing $m$}
    \label{alg::piececonst}
        \begin{algorithmic}
        \STATE{\textbf{Input}: One-dimensional distribution $\pi$ for $\theta$, representation of $x$ as a piecewise constant function with $m$ pieces.}
        \STATE{\textbf{Output}: A representation of $x$ as a piecewise constant function with $m + 1$ pieces}
        \STATE{Let $x = \sum_i c_i \mathcal{X}_i$, where $\mathcal{X}_i$ is an indicator function over some continuous interval, and intervals are sorted in ascending order.}
        \STATE{Let $\theta'$ be a sample from the distribution $\pi$.}
        \STATE{Let $I$ be the index such that $\theta'$ falls in the interval $\mathcal{X}_I$.}
        \STATE{Subdivide the interval indexed by $\mathcal{X}_I$ into two continuous disjoint intervals with the breakpoint given by $\theta'$, and with $x$ taking the value $c_I$ on both new intervals}
        \STATE{Return $x$.}
        \end{algorithmic}
\end{algorithm}

In Algorithm \ref{alg::piececonst}, we detail what happens when we increase $m_k$. In the remainder of this subsection, we detail some results on the convergence of RSGUQ using this method. We note first that we assume throughout that the function $f$ is defined on a bounded domain. This is in contrast to previous Sections, but true for our submodular example. The first result follows in the same way as Theorem \ref{theorem::subgradientfixed}:

\begin{lemma} \label{lem::PC_Fixed}
    Suppose that for the duration of one outer loop, the piecewise constant representation of $x$ is fixed. Let $f^*$ be the optimal function value when the input is restricted to this representation. In our one outer loop, $t \geq \frac{\alpha^2 G^2}{\rho_{\varepsilon}^2}$ and $K = \lceil \log_{\alpha} \left( \frac{\varepsilon_0}{\varepsilon}\right) \rceil$. Then we obtain a 2$\varepsilon$ accurate solution at the end of the loop.
\end{lemma}

The next result is an analogue to Lemma \ref{lemma::LinConvRemainder}, and follows from the fact that $f$ is Lipschitz continuous with constant $L$. Before this however, we need to define precisely what truncated optima are in this setting. To do so, we first precisely define how we partition the domain of $\Theta$:
\begin{definition}
    Let $\Theta$ be a one-dimensional bounded interval from $a$ to $b$ (possibly not including endpoints) that serves as the domain of $\theta$. We call a sequence $\mathcal{C} = (c_1,\ldots,c_n)$ a partition of length $n+1$, and it corresponds to the partition of this interval into the parts $(a,c_1), (c_1,c_2), \ldots, (c_n, b)$ (the choice of what to do with endpoints is arbitrary.
    
    A partition $\mathcal{C}_1$ is said to be finer than another partition $\mathcal{C}_2$, if every interval induced by the partition $\mathcal{C}_1$ is included in an interval induced by $\mathcal{C}_2$, possibly excluding endpoints.
\end{definition}
\begin{lemma}
    Let $\mathcal{C}$ be a partition of length $n+1$, and let $f^*_C(\theta)$ be the optimal function value under this partition. Further, let $f^*(\theta)$ be the true optimal function value.
    
    Then on each interval $\mathcal{I}$ induced by $\mathcal{C}$, the value of $f^*_C(\theta)$ for $\theta \in \mathcal{I}$ is the average value of $f^*(\theta)$ over $\mathcal{I}$.
\end{lemma}
\begin{proof}
    The result is clear on considering the equivalent quantity to be minimised, where $c$ is a constant:
    \begin{equation*}
        \int_{\mathcal{I}} (c - f^*(\theta))^2 \pi(d\theta), 
    \end{equation*}
    with the result follows by differentiating under the integral sign, then noting the $f^*$ term will integrate to the average value.
\end{proof}
From here, we can now give an analogue to Lemma \ref{lemma::LinConvRemainder}:
\begin{lemma} \label{lem::PC_OptimumConvergence}
    Let $M$ be the value of the element of the subdifferential of $f$ throughout its domain with maximal absolute value. Let $\mathcal{C}$ have an interval with maximum $\pi$-measure $\mu$. Then we have:
    \begin{equation*}
        ||f(x(\theta),\theta) - f^*_n|| \leq \mathcal{O}(\mu)
    \end{equation*}
\end{lemma}
\begin{proof}
    The result follows by noting that $f$ is $L$-lipschitz, and intervals can have at most $\mu$ measure, which bounds the error in each interval.
\end{proof}
To finish this discussion, we would like a control on the maximum measure. We state the following result without loss of generality assuming $\Theta = [0,1]$, and that $\theta$ is uniformly distributed:
\begin{lemma}
    Suppose while running Algorithm \ref{alg::piececonst}, we have drawn $n$ uniform variables $\theta_1 \leq \ldots \leq \theta_n$. Let $W_1 = \theta_1$ and $W_i = \theta_i - \theta_{i-1}$, $i > 1$. Further, let $X_1,\ldots, X_{n+1}$ be $n+1$ standard exponential variables, and let $X_M$ be the maximum of these. Then the expected measure of the piece with maximum measure is distributed as:
    \begin{equation*}
        E_n = \mathbb{E}\left(\max_i W_i\right) = \frac{H_{n-1}}{n+1}
    \end{equation*}
    where $H_n$ is the $n$th harmonic number. We therefore have:
    \begin{equation*}
        E_n = \mathcal{O}\left( \frac{\log n}{n}\right)
    \end{equation*}
\end{lemma}
A proof can be found in \cite{pinelis2019order}.

With these results, we can state a main convergence result, which follows similarly to Theorem \ref{theorem::mainresult}.
\begin{theorem} \label{lem::PC_OverallConv}
    Suppose that for each outer loop, we have estimates of $G, V$ that hold throughout the loop. Furthermore, for each inner loop where we use $m_k$ iterations, let $\rho = \min_{m_1,\ldots,m_k} \left(\rho_{\varepsilon,m_k}\right)$. Then let $t_i \geq \frac{\alpha^2 \left(G^2 + V^2\right)}{\rho^2}, K_i = \lceil \log_\alpha \left( \frac{\varepsilon_0}{\varepsilon} \right) \rceil$.
    
    Let $i^*$ be the first outer loop in which the partition $\mathcal{C}$ is such that $||f^*_C(\theta) - f^*||_\pi \leq \varepsilon$. Then we obtain a $3\varepsilon$-accurate solution at the end of the outer loop $i^*$.
\end{theorem}

\section{Experiments}
In this section we provide some numerical results of the RSGUQ method. We begin by evaluating the method on a simple unconstrained problem, before comparing some different methods for a min-cut graph problem. We will see that for each problem we consider, that we get simple bounds for $B_{\varepsilon,m}$ that make selecting the algorithms parameters easy.
\subsection{A simple constrained problem}
In this subsection, we minimise a simple quadratic in two variables defined as follows:

\begin{equation}
f(x(\theta), y(\theta))=\begin{cases}
          \frac{\mu}{4}\left(x(\theta) - x^*(\theta)\right)^2 + \frac{L}{2}\left(y(\theta) - y^*(\theta)\right)^2  \quad &\text{if} \, x(\theta) > x^*(\theta) \text{ and } y(\theta) > y^*(\theta) \\
          \frac{\mu}{2}\left(x(\theta) - x^*(\theta)\right)^2 + \frac{L}{2}\left(y(\theta) - y^*(\theta)\right)^2  \quad &\text{if} \, x(\theta) < x^*(\theta) \text{ and } y(\theta) > y^*(\theta) \\
          \frac{\mu}{4}\left(x(\theta) - x^*(\theta)\right)^2 + \frac{L}{4}\left(y(\theta) - y^*(\theta)\right)^2  \quad &\text{if} \, x(\theta) > x^*(\theta) \text{ and } y(\theta) < y^*(\theta) \\
          \frac{\mu}{2}\left(x(\theta) - x^*(\theta)\right)^2 + \frac{L}{4}\left(y(\theta) - y^*(\theta)\right)^2  \quad &\text{if} \, x(\theta) < x^*(\theta) \text{ and } y(\theta) < y^*(\theta), \\
     \end{cases} \label{eqn::QuadraticRSG}
\end{equation}
where we have:
\begin{equation*}
    x^*(\theta) = y^*(\theta) = |(4/5 + 1/4 \exp(\sin(\theta)) - \cosh(\sin(\theta)^2)|(1+\sin(2\theta)).
\end{equation*}
This optimum function also appears in \cite{crepey2020uncertainty} and \cite{mcmeel2021uncertainty}, and is chosen for the fact that its basis representation decays slowly. We also impose the condition on our solution $w(\theta) = (x(\theta),y(\theta))$ that $||w||_\pi \leq 1.5$. As explained before, this is equivalent to an $L_2$ constraint on the vector of coefficients, so can be projected easily. Furthermore, benchmark numerical integration tells us that $||x^*(\theta)||_\pi = ||y^*(\theta)||_\pi < 0.52$, so the solution to this problem will be the correct one.

We note the theoretical epoch sizes as derived in \ref{theorem::mainresult} can be very large: in general we know that we have $G^2 \leq L^2$ as $f$ is Lipschitz continuous, but cannot improve upon this in general without special structure. In these experiments we show that epochs of much smaller size still ensure convergence.

Rather than trying to identify the $B_{\varepsilon}$ coefficients for this function, we note that we can put the following simple bound on the values of $f$, where we write $(x(\theta), y(\theta)) = w(\theta)$ to keep similarity with Definition \ref{defn::localerrorbound}:
\begin{equation*}
    ||w(\theta) - w^*(\theta)||_\pi \leq \sqrt{\frac{4}{L}} ||f(w(\theta),\theta) - f(w^*(\theta),\theta)||_\pi
\end{equation*}
where we note that $f(w^*(\theta),\theta) = 0$ $\forall$ $\theta$. Therefore we have the bound $B_{\varepsilon,k} \leq \frac{2\varepsilon}{\sqrt{L}}$. Due to this bound, we see that Corollary \ref{cor::NewLinear} will actually give a linear rate of convergence.


\begin{figure}
		\includegraphics[width=0.7\linewidth]{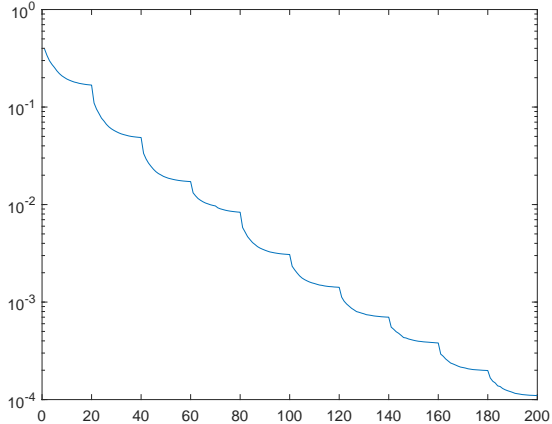}
		\caption{A plot of function error against number of calls to Algorithm \ref{alg::subroutineRSG} in minimising Equation \eqref{eqn::QuadraticRSG}.} \label{fig::QuadError}
\end{figure}

With this in mind, we set $t_i = 50, K_i = 20$, and we terminate after $i = 10$ loops. We have also set $\mu = 1, L = 50$. In Figure \ref{fig::QuadError} we see a plot of the error in the function value. We note that the $x$-axis is numbered by the amount of times the SG inner subroutine from \ref{alg::subroutineRSG} is run. With that in mind, we see that each cusp in the plot represents one inner loop in the RSG routine. A linear rate of convergence, as expected, is clearly observed.

With this proof of concept, we now move on to dealing with a more interesting application: solving a discrete problem (min-cut) via its convex relaxation.

\subsection{Min-Cut Problem}
In this subsection we deal with the problem of solving a min-cut problem. As stated previously, we will use the tools of submodular optimisation to solve this problem. We first detail the specific example we wish to solve:
\begin{example} \label{exmp::IntroConvex}
		Consider the minimum $s,t$ cut problem from Figure $1$ with $\theta \geq 0$, where we have a graph $G = (V,E)$ with a source $s$ and sink $t$ that we require to be in different subsets. The problem statement asks us to find sets $S,V\setminus S$, with $s \in S$ and $t \in V \setminus S$ such that:
		\begin{equation}
		    C(S) = \sum_{i \in S, j \in V \setminus S, (i,j) \in E} w_{ij}
		\end{equation}
		is minimised, where $w_{ij}$ is the weight of the edge between $i, j$.
		
		It can be shown that this problem is a submodular minimisation problem, and further that it can be minimised by casting it as a convex minimisation problem via the Lovasz Extension\cite{fujishige2005submodular}, a common convex relaxation used to minimise submodular functions. More details about the Lovasz Extension can be found in \cite{fujishige2005submodular}. In our case, it can be shown that the Lovasz Extension is equal to:
		\begin{equation} \label{eqn::Lovasz}
		    f(x_1,x_2) = 2|x_1-x_2| + \theta x_1 + 3|1-x2|
		\end{equation}
		with $x_1, x_2 \in [0,1]$. We see that the set of minima depend on the value of $\theta$, and in particular, that set $M(\theta)$ can be shown to be:
		\begin{equation}
		    M(\theta) = \{(0,1), \theta \geq 2\}, \quad \{(1,1), \theta \leq 2\}.
		\end{equation}
	\end{example}

    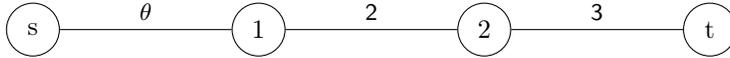
\begin{figure} \label{fig::stCutIntro}	
    	\begin{tikzpicture}[
          vertex/.style = {shape=circle,draw,minimum size=2em},
          edge/.style = {->,-Latex},
          ]
          \node[vertex] (s) at (2,0) {s};
          \node[vertex] (1) at (5,0) {1};
          \node[vertex] (2) at (8,0) {2};
          \node[vertex] (t) at (11,0) {t};
          \path[every node/.style={font=\sffamily\small}]
            (s) edge  node[pos=0.5,above] {$\theta$} (1)
            (1) edge node[pos=0.5,above] {2} (2)
            (2) edge node[pos=0.5,above] {3} (t);
        \end{tikzpicture}
        \caption{A min $s-t$ cut problem found in Example \ref{exmp::IntroConvex}}
    \end{figure}

We note that the problem set-up will require us to impose constraints of the form $0 \leq x(\theta) \leq 1$ $\forall$ $\theta$, which are highly non-trivial for most common orthogonal series. With this in mind, as discussed in Section \ref{miniSec::PiecewiseConst}, we will divide the one-dimensional range of $\theta$ into an increasingly finer piecewise constant function. Note that in this case, the projection is trivial as each piece can be projected to $[0,1]$. We expect \textit{a priori} that the optimum $x^*(\theta)$, will be piecewise constant, so this method is suited to our problem.
For our experiments we use the : $K = 20$, $t = 50$, $\alpha = 1.2$, and an initial step size of $0.01$. We perform a total of $10$ outer loops. At inner loop $j$, we use $\left(j+10\right)^{0.8} + 10$ basis functions.


In Figure \ref{fig::SubPiecewise} we present the results. The $x$-axis is indexed by the number of calls of Algorithm \ref{alg::subroutineRSG}, and the $y$-axis the error $||f(x(\theta),\theta) - f^*||_\pi^2$. We see that linear convergence is attained overall, with characteristic cusps in each outer loop as the step size is being reduced. We also note that any apparent slowing down in convergence rate can be explained by not having sufficient basis functions.

\begin{figure}
		\includegraphics[width=0.7\linewidth]{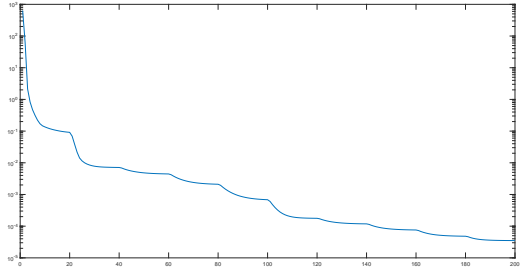}
		\caption{A plot of function error against number of calls to Algorithm \ref{alg::subroutineRSG} in minimising Equation \eqref{eqn::Lovasz}.} \label{fig::SubPiecewise}
\end{figure}

\section{Conclusion}
In this paper we introduced a subgradient method equivalent to the stochastic uncertainty quantification method given by \cite{crepey2020uncertainty}, and the gradient descent method given by \cite{mcmeel2021uncertainty}. The latter work only considers the case where $f$ is strongly convex, and here we relax to the case where $f$ is non-strongly convex and nonsmooth. We motivated our method by considering that to estimate the statistics of $x^*(\theta)$ naively, we would need to perform subgradient descent an exponential number of times.
	
We showed that each method converges with rate $1/\varepsilon^2$ to a solution, and as the number of basis functions grows, we eventually converge to a $\varepsilon$ neighbourhood of the true solution. Numerical evidence supported the above claims.
	
We then considered a subclass of problems called \textit{polyhedral convex optimisation}, which includes piecewise linear functions. We showed that a linear rate of convergence is obtained in this case when the number of basis functions is high enough. We showed that submodular minimisation is a polyhedral convex optimisation problem, and evaluated the performance of our method on a min $s-t$ cut problem.

\newpage
\bibliographystyle{plain}
\bibliography{UQVersionCurrent}

\newpage
\appendix
\section{Fixed Level Convergence Proofs}
In this section, we give the proofs for Lemma \ref{lem::constStepSG} and Theorem \ref{theorem::subgradientfixed}.
\subsection{Proof for Lemma \ref{lem::constStepSG}}
We begin by writing the statement of the Lemma again:
\begin{lemma}
    Suppose we run the subroutine Algorithm \ref{algo_sub}, but we only have noisy subgradients as defined above. Using a constant step size $\eta$ for $T$ iterations, we have:
    \begin{equation*}
        \mathbb{E}\left(\mathbb{E}_\pi\left(f(\tilde{x}_T(\theta),\theta) - f(x^*(\theta),\theta) | \mathcal{F}_k\right)\right) \leq \frac{\left(G^2 + V^2\right) \eta}{2} + \frac{||x_1(\theta) - x(\theta)||_\pi^2}{2\eta T}
    \end{equation*}
    where $\tilde{x}_T = \frac{1}{T}\sum_{i=1}^T x_i$, the outer expectation is with respect to the filtration $\mathcal{F}_k$, and $x^*(\theta)$ is some truncated optimum.
\end{lemma}
Before we prove this, we first prove a couple of auxiliary results. First, we need a Lemma of our own on the subgradient that is a simple consequence of its definition:
\begin{lemma}
    Let $g_m(\theta)$ be a truncated subgradient. Then for all $x(\theta), y(\theta) \in L^2_{\pi,m}$ we have:
    \begin{equation*}
        \mathbb{E}_\pi\left(f(x(\theta),\theta) - f(y(\theta),\theta)\right) \geq \langle g_x(\theta), x(\theta) - y(\theta) \rangle_\pi
    \end{equation*}
\end{lemma}

We now also adapt Lemma $1$ of \cite{yang2018rsg}: 
\begin{lemma} \label{lem::wTechLem}
    For any $\varepsilon > 0$ such that $\mathcal{L}_\varepsilon \neq 0$ and any $w \in \Omega$, we have:
    \begin{equation*}
        ||w(\theta) - w^\dagger_{\varepsilon}(\theta)||_\pi \leq \frac{1}{\rho_{\varepsilon}} \mathbb{E}_\pi\left(f(w(\theta),\theta) - f(w^{\dagger}_{\varepsilon}(\theta))\right)
    \end{equation*}
\end{lemma}
\begin{proof}
    The proof proceeds similarly to \cite{yang2018rsg}, except instead of the Equation after Equation $(8)$ there, we use our definition of convexity and the subgradient, thus giving:
    \begin{equation*}
        \zeta\mathbb{E}_\pi\left(f(w(\theta),\theta) - f(w^{\dagger}_{\varepsilon}(\theta))\right) \geq \zeta \langle w(\theta) - w^\dagger_{\varepsilon}(\theta), g(\theta) \rangle_\pi
    \end{equation*}
    and the remainder of the proof follows as in \cite{yang2018rsg}.
\end{proof}
Now that we have done this, we proceed with the proof of Lemma \ref{lem::constStepSG}.
\begin{proof}[Proof of Lemma \ref{lem::constStepSG}]
At iteration $k$, we have the iterate $x^k(\theta)$, the step size $\alpha_k$, and the vector of basis coefficients of the subgradient as $g_k(\theta)$. We label an optimum as $x^*(\theta)$. We define the step before projection as:
\begin{equation*}
    z^{k+1}(\theta) = x^k(\theta) - \alpha_k \tilde{g}_k(\theta)
\end{equation*}
Knowing this, we can write:
\begin{align*}
    &\mathbb{E}\left( ||z^{k+1}(\theta) - x^*(\theta)||_\pi^2 | \mathcal{F}_k\right) = \mathbb{E}\left(||x^k(\theta) - \alpha_k \tilde{g}_k - x^*(\theta)||_\pi^2 | \mathcal{F}_k\right) \nonumber \\
    &= \mathbb{E}\left(||x^k(\theta) - x^*(\theta)||_\pi^2 | \mathcal{F}_k\right) -2\alpha_k\mathbb{E}\left(\langle \tilde{g}_k, x^k(\theta)-x^*(\theta) \rangle_2 | \mathcal{F}_k\right) + \alpha_k^2\mathbb{E}\left(||\tilde{g}_k||_\pi^2 | \mathcal{F}_k \right) \nonumber \\
    &= \mathbb{E}\left(||x^k(\theta) - x^*(\theta)||_\pi^2 | \mathcal{F}_k\right) -2\alpha_k \langle \mathbb{E}(\tilde{g}_k | \mathcal{F}_k), x^k(\theta)-x^*(\theta) \rangle_2 + \alpha_k^2\mathbb{E}\left(||\tilde{g}_k||_\pi^2 | \mathcal{F}_k \right) \nonumber \\
    &\leq \mathbb{E}\left(||x^k(\theta) - x^*(\theta)||_\pi^2 | \mathcal{F}_k\right) -2\alpha_k \mathbb{E}\left(\mathbb{E}_\pi\left(f(x_k(\theta),\theta) - f(x^*(\theta),\theta) | \mathcal{F}_k\right) \right) + \alpha_k^2\mathbb{E}\left(||\tilde{g}_k||_\pi^2 | \mathcal{F}_k \right),
\end{align*}
where in the last line, the fact that $\tilde{g}_k$ is an unbiased estimate of a subgradient. By the properties of projections, if $x^{k+1}(\theta) = \mathcal{P}(x^k(\theta))$, we have
\begin{equation*}
\mathbb{E}\left( ||x^{k+1}(\theta) - x^*(\theta)||_\pi^2 | \mathcal{F}_k\right) \leq \mathbb{E}\left( ||z^{k+1}(\theta) - x^*(\theta)||_\pi^2 | \mathcal{F}_k\right)
\end{equation*}
Then combining that with the previous inequality, we can say:
\begin{align*}
    \mathbb{E}\left( ||x^{k+1}(\theta) - x^*(\theta)||_\pi^2 |\mathcal{F}_k \right) &\leq \mathbb{E}\left(||x^k(\theta)
    - x^*(\theta)||_\pi^2 | \mathcal{F}_k\right) \\ 
    &-2\alpha_k \mathbb{E}\left(\mathbb{E}_\pi\left(f(x_k(\theta),\theta) - f(x^*(\theta),\theta) | \mathcal{F}_k\right) \right) \\
    &+ \alpha_k^2 \left( V^2 + G^2 \right),
\end{align*}
Where we have also used $\mathbb{E}||g_k^2||_\pi^2 \leq G^2$. We can now iterate this inequality through $k$ and get:
\begin{align*}
    \mathbb{E}\left( ||x^{k+1}(\theta) - x^*(\theta)||_\pi^2 |\mathcal{F}_k \right) &\leq \mathbb{E}||x^1(\theta) - x^*(\theta)||_\pi^2 \\
    &-\sum_{i=1}^k 2\alpha_i \mathbb{E}\left(\mathbb{E}_\pi\left(f(x_i(\theta),\theta) - f(x^*(\theta),\theta) | \mathcal{F}_i\right) \right) \\
    &+ \sum_{i=1}^k \alpha_i^2 \left(G^2 + V^2\right)
\end{align*}
By dropping positive terms, setting $k = T$, and setting $\alpha_i = \eta$, $\forall i$, we find:
\begin{equation*}
     \frac{1}{T}\sum_{i=1}^T\mathbb{E}\left(\mathbb{E}_\pi\left(f(x_i(\theta),\theta) - f(x^*(\theta),\theta) | \mathcal{F}_i\right) \right) \leq \frac{||x^1(\theta) - x^*(\theta)||_\pi^2 + \eta^2\left(G^2 + V^2\right)}{2\eta T}
\end{equation*}
from which we get the result using Jensen's inequality.
\end{proof}

\subsection{Proof for Theorem \ref{theorem::subgradientfixed}}

\begin{proof}
    We use an inductive framework similar to \cite{yang2018rsg}. We note that for the duration of this proof, we remove the $m$ subscript on relevant quantities for clarity. Firstly we recall the point $x_{k,\varepsilon}^\dagger$ as being the closest point to $x_k$ in the $\varepsilon$ sublevel set. Then let $\varepsilon_k = \frac{\varepsilon_0}{\alpha^k}$, which gives $\eta_k = \frac{\varepsilon_k}{\left(G^2 + V^2\right)}$. By induction, we'll show:
    \begin{equation}
        \mathbb{E}_\pi\left(f(x_k(\theta),\theta) - f^*_k\right) \leq \varepsilon_k + \varepsilon \label{equation::ToInduct}
    \end{equation}
    for $k = 0,1,\ldots,K$, which gives our result for $k = K$. The base case of $k = 0$ is clear, so we suppose it holds for $k-1$. we can apply Lemma \ref{lem::constStepSG} to the $k$th stage and find:
    \begin{equation}
        \mathbb{E}_\pi\left(f(x_k(\theta),\theta) - f(x^{\dagger}_{k-1, \varepsilon}(\theta),\theta)\right) \leq \frac{\left(G^2 + V^2\right) \eta_k}{2} + \frac{||x_{k-1}(\theta) - x^{\dagger}_{k-1, \varepsilon}(\theta)||_\pi^2}{2\eta_k t} \label{equation::THMPrelim}
    \end{equation}
    After this we separate into cases: firstly where $x_{k-1} \in \mathcal{S}_\varepsilon$. In this case $x^{\dagger}_{k-1, \varepsilon} = x_{k-1}$, and then from the previous Equation we see:
    \begin{equation*}
        \mathbb{E}_\pi\left(f(x_k(\theta),\theta) - f(x^{\dagger}_{k-1, \varepsilon}(\theta),\theta)\right) \leq \frac{\left(G^2 + V^2\right) \eta_k}{2} = \frac{\varepsilon_k}{2} 
    \end{equation*}
    From this we see that:
    \begin{equation*}
        \mathbb{E}_\pi\left(f(x_k(\theta),\theta) - f^*\right) \leq \mathbb{E}_\pi\left(f(x_k(\theta),\theta) - f(x^{\dagger}_{k-1, \varepsilon}(\theta),\theta)\right) + \mathbb{E}_\pi\left(f(x^{\dagger}_{k-1, \varepsilon}(\theta),\theta) - f^*\right) \leq \frac{\varepsilon_k}{2} + \varepsilon
    \end{equation*}
    We now consider the second case, where $x_{k-1} \in \mathcal{S}_\varepsilon$, which implies that $\mathbb{E}||f(x_k(\theta),\theta) - f(x^{\dagger}_{k-1, \varepsilon}(\theta),\theta)||_\pi = \varepsilon$. From Lemma \ref{lem::wTechLem}, we have:
    \begin{align*}
        &||x_{k-1}(\theta) - x^\dagger_{k-1,\varepsilon}(\theta)||_\pi \nonumber \\
        &\leq \frac{1}{\rho_{\varepsilon}} \mathbb{E}_\pi\left(f(x_{k-1}(\theta),\theta) - f(x^\dagger_{k-1,\varepsilon}(\theta))\right) \nonumber \\
        &\leq \frac{\mathbb{E}_\pi\left(f(x_{k-1}(\theta),\theta) - f^*\right) + \mathbb{E}_\pi\left(f^* - f(x^\dagger_{k-1,\varepsilon}(\theta))\right)}{\rho_{\varepsilon}} \nonumber \\
        &\leq \frac{\varepsilon_{k-1} + \varepsilon - \varepsilon}{\rho_{\varepsilon}} \nonumber \\
        &\leq \frac{\varepsilon_{k-1}}{\rho_{\varepsilon}}.
    \end{align*}
    We can combine this inequality with $\eta_k = \frac{\varepsilon_0}{\alpha^k}$ and the bound on $t$ to find:
    \begin{equation}
        \frac{||x_{k-1}(\theta) - x^{\dagger}_{k-1, \varepsilon}(\theta)||_\pi^2}{2\eta_k t} \leq \frac{\varepsilon_{k-1}^2}{2\varepsilon_k \alpha^2} = \varepsilon_k \label{equation::THMSecond}
    \end{equation}
    Which we then combine with Equation \eqref{equation::THMPrelim} to find:
    \begin{equation*}
         \mathbb{E}\left(\mathbb{E}_\pi\left(f(x_k(\theta),\theta) - f(x^{\dagger}_{k-1, \varepsilon}(\theta),\theta)\right)\right) \leq \varepsilon_k
    \end{equation*}
    and then combine with the fact that $\mathbb{E}\left(\mathbb{E}_\pi\left(f^*_k - f(x^{\dagger}_{k-1, \varepsilon}(\theta),\theta)\right)\right) = \varepsilon$ to show Equation \eqref{equation::ToInduct} holds for all $K = 0,1,\ldots,K$.
    
    By the definition of $K$, we have that $\varepsilon_K \leq \varepsilon$, and so writing Equation \eqref{equation::ToInduct} for $k = K$ we find:
    \begin{equation*}
       \mathbb{E}\left( \mathbb{E}_\pi\left(f(x_K(\theta),\theta) - f^m_*\right) \right) \leq \varepsilon_K + \varepsilon \leq 2\varepsilon
    \end{equation*}
    as required.
\end{proof}

\subsection{Proof for Lemma \ref{lem::Bep}}
\begin{proof}
    The proof proceeds similarly to \cite{yang2018rsg}, except instead of their first Equation, we use our definition of convexity, subgradients, and the normal cone, thus giving:
    \begin{equation}
        \mathbb{E}_\pi\left(f(u^*(\theta),\theta) - f(u(\theta)\theta)\right) \geq \langle u^*(\theta) - u(\theta), g_u(\theta) + v_u(\theta) \rangle_\pi
    \end{equation}
    where $u \in \mathcal{L}_{\varepsilon}$, $u^*$ the closest optimal point to $u$, $g$ is a subgradient of $u$, and $v$ a vector in the normal cone. This implies, on taking Cauchy-Schwarz:
    \begin{equation}
        ||u^*(\theta) - u(\theta)||_\pi ||g_u(\theta) + v_u(\theta)|| \geq \mathbb{E}_\pi\left(f(u(\theta),\theta) - f(u^*(\theta)\theta)\right) = \varepsilon
    \end{equation}
    By the fact that $u$ is in $\mathcal{L}_\varepsilon$. After this, the remainder of the proof follows as in \cite{yang2018rsg}.
\end{proof}

\section{Overall Convergence Rate Proofs}
In this section, we give a proof for Theorem \ref{theorem::mainresult}. Before that, we need to discuss the set $\mathcal{S}_\varepsilon$ and the quantity $\rho_{\varepsilon}$ that we have been working with so far. We note that they also depend on the number of basis functions in the current loop
\subsection{Proof for Theorem \ref{theorem::mainresult}}
\begin{proof}
    The proof is very similar to that of Theorem \ref{theorem::subgradientfixed}, but the quantity we aim to bound by induction is slightly different, namely:
    \begin{equation*}
        \mathbb{E}_\pi\left( f(x_k(\theta),\theta) - f^*_k\right) \leq \varepsilon_k + \varepsilon
    \end{equation*}
    where $f^*_k$ is the optimum at $m_k$ basis functions. For $k = K_{i^*}$, this implies the Theorem by noting that we have $||f^*_k - f^*||_\pi \leq \varepsilon$ by the hypotheses of the Theorem.
    
    We consider the inductive step in Theorem \ref{theorem::subgradientfixed}. In the first case, where $x_{k-1} \in \mathcal{S}_{\varepsilon,k-1}$, the proof proceeds as in that Theorem. In the second case, it also proceeds similarly. The only delicate step to consider is the matter of $\rho$. If we use the value $\rho_n$ for $\rho$ in inner loop $n$, then Equation \eqref{equation::THMSecond} will read:
    \begin{equation*}
        \frac{||x_{k-1}(\theta) - x^{\dagger}_{k-1, \varepsilon}(\theta)||_\pi^2}{2\eta_k t} \leq \frac{\rho_{n+1}}{\rho_n}\frac{\varepsilon_{k-1}^2}{2\varepsilon_k \alpha^2} = \frac{\rho_{n+1}}{\rho_n}\varepsilon_k.
    \end{equation*}
    However, by the definition of $\rho$, this first factor is smaller than one, and therefore we can still find the bound:
    \begin{equation*}
         \mathbb{E}\left(\mathbb{E}_\pi\left(f(x_k(\theta),\theta) - f(x^{\dagger}_{k-1, \varepsilon}(\theta),\theta)\right)\right) \leq \varepsilon_k
    \end{equation*}
    and thus we prove the Theorem by following the remainder of Theorem \ref{theorem::subgradientfixed}.
\end{proof}

\end{document}